\documentclass[12pt,reqno]{amsart}
\usepackage[textsize=Small]{todonotes}  
\allowdisplaybreaks  

\usepackage{amssymb}
\usepackage[all]{xy}
\usepackage{enumerate}

\newcommand{\C}{{\mathbb C}}

\newcommand{\Z}{{\mathbb Z}}

\newcommand\superpuesto[2]{\genfrac{}{}{0pt}{}{#1}{#2}}

\newcommand\diag{\operatorname{Diag}}
\newcommand\End{\operatorname{End}}

\newcommand\GL{{\operatorname{GL}}}
\newcommand\Hom{\operatorname{Hom}}
\newcommand\Id{\operatorname{Id}}

\renewcommand{\sl}{\mathfrak{sl}}

\newtheorem*{theo}{Theorem}
\numberwithin{equation}{section}

\newtheorem{defn}[equation]{Definition}
\newtheorem{thm}[equation]{Theorem}
\newtheorem{cor}[equation]{Corollary}
\newtheorem{lem}[equation]{Lemma}
\newtheorem{prop}[equation]{Proposition}

\theoremstyle{remark}
\newtheorem{rem}[equation]{Remark}
\theoremstyle{remark}

\begin{document}

\title[Construction of simple non-weight $\sl(2)$-modules]
{Construction of simple non-weight $\sl(2)$-modules of arbitrary rank}
\author[F. J. Plaza Mart\'{\i}n]{F. J. Plaza Mart\'{\i}n}
\author[C. Tejero Prieto]{C. Tejero Prieto}

\address{Departamento de Matem\'aticas AND IUFFyM, Universidad de
Salamanca,  Plaza de la Merced 1-4
        \\
        37008 Salamanca. Spain.
        \\
         Tel: +34 923294460. Fax: +34 923294583}
\date\today
\thanks{
\indent This work is supported by the research contracts MTM2013-45935-P and
MTM2012-32342 of Ministerio de Ciencia e Innovaci\'{o}n, Spain.  \\
}
\email{fplaza@usal.es}\email{carlost@usal.es}

\begin{abstract}
We study simple non-weight $\sl(2)$-modules which are finitely generated as $\C[z]$-modules. We show that they are described in terms of semilinear endomorphisms and prove that the Smith type induces a stratification on the set of these $\sl(2)$-modules, providing thus new invariants. Moreover, we show that there is a notion of duality for these type of $\sl(2)$-modules. Finally, we show that there are  simple non-weight $\sl(2)$-modules of arbitrary rank by constructing a whole new family of them. 
\end{abstract}

\maketitle
\setcounter{tocdepth}{1}

\section{Introduction}

The study of simple $\sl(2)$-modules can be divided into two big classes: weight modules and non-weight modules. While simple weight modules have been exhaustively classified and explicitly described (e.g. \cite{Mazor}), the study of simple non-weight modules is much more difficult and subtle (\cite{Block, Bavula2}, see also \cite{Arnal, Bavula1}). In particular, Bavula proved  (\cite{Bavula1, Bavula2}) that there exists a bijection between simple non-weight modules and irreducible elements of a certain non-abelian euclidean algebra. Nevertheless, the lack of an explicit construction or classification of such irreducible elements has prevented the construction of large families of  simple non-weight modules. To our best knowledge, in the literature there are only a few explicit examples of simple non-weight modules: those obtained by Bavula (\cite[Corollary 3.9 b]{Bavula1}, \cite[Corollary 2, pag. 1046]{Bavula2}),  which generalize the Arnal-Pinczon series (\cite{Arnal}) and include as  a particular case the Whittaker modules (\cite{Kostant}), and the modules built by Puninski{\u\i} (\cite[Proposition 1]{Puninskii}). With the exception of Whittaker modules, these simple non-weight $\sl(2)$-modules are not finitely generated with respect to their natural $\C[z]$-module structure.

In this paper we endeavour to study the class of  simple non-weight $\sl(2)$-modules that are finitely generated  as $\C[z]$-modules and we provide an explicit construction of a large family of them. 
Let us briefly describe our main results.

Let $\sl(2)$ be the Lie algebra of the Lie group $\operatorname{SL}(2,\C)$. We have set $\C$ as the base field but everything admits a straightforward generalization to an arbitrary algebraically closed field of characteristic $0$. Let $\{e,f,h\}$  be a Chevalley basis satisfying the commutation relations:
    \begin{equation}\label{eq:basis-sl2}
    [e,f]=h\; , \qquad
    [h,e]=2e \; , \qquad
    [h,f]=-2f\; .
    \end{equation}
Let $V$ be an  $\sl(2)$-module. We say that $V$ is a Casimir module of semi-level $\mu$ if the Casimir operator acts by the homothety $(2\mu+1)^2\in\C $ with ${\mathfrak{Re}}(\mu)\geq -\frac12$. 

Every $\sl(2)$-module $V$ has a natural $\C[z]$-module structure where $z$ acts by $-\frac12 h$.  
As a matter of convention in this paper we say that an $\sl(2)$-module $V$ is torsion free or finitely generated if it is so with respect to its natural  $\C[z]$-module structure.  

As a first result, we prove that any simple non-weight $\sl(2)$-module $V$ is a finite rank torsion free module with respect to its natural $\C[z]$-module structure, see Theorem \ref{thm:dichotomy} and Lemma \ref{lem:simpletorsionfree-finiterank}. This implies that any such $V$ is a $\C[z]$-submodule of a finite dimensional $\C(z)$-vector space. Therefore, one is naturally led in this way to consider rational $\sl(2)$-modules; that is, $\sl(2)$-representations on finite dimensional vector spaces over the field $\C(z)$ of rational functions. Moreover, if in addition we assume that the non-weight $\sl(2)$-module $V$ is a finitely generated $\C[z]$-module then we prove that it is a polynomial $\sl(2)$-module; that is, there is an isomorphism $V\simeq \C[z]^n$ as $\C[z]$-modules for a certain integer $n$.


More precisely, the relevance of polynomial representations is unveiled by Theorem~\ref{thm:simpletorsionfreefingen-finiterankfree} which proves  that

	{\small $$
	\left\{\begin{gathered}
	\text{Simple non-weight finitely}
	\\
	\text{generated $\sl(2)$-modules}
	\end{gathered}\right\}
	\,\subseteq\, 
	\coprod_{\superpuesto{\mu\in\C}{\frak{Re}(\mu)\geq -\frac12}} 
	\left\{\begin{gathered}
	\text{Polynomial Casimir}
	\\
	\text{$\sl(2)$-modules of }
	\text{semi-level $\mu$}
	\end{gathered}\right\}.
	$$}Therefore, the study of simple non-weight finitely generated  $\sl(2)$-modules reduces to the study of the sets $\sl(2)_\mu\mathrm{-Mod}(V)$ of polynomial Casimir representations of semi-level $\mu$, where $V$ is a free $\C[z]$-module of rank $n$.

This set is described in terms of semilinear endomorphisms of $V$, i.e. those $\varphi\in\End_{\C}(V)$ such that $\varphi(z\cdot v)=(z+1) \varphi(v)$, as follows

\begin{theo}[see Theorem~\ref{thm:polynomialmuCasimir} for the precise statement]
Let $V$ be a rank $n$ free $\C[z]$-module. There is an identification 
	{\small $$
	\sl(2)_\mu\mathrm{-Mod}(V)\,\simeq\, \left\{
	\begin{gathered}
	\varphi\in \End_{\C}(V) \text{ s.t. $\varphi$ is semilinear and the}
	\\
	\text{$n$-th invariant  factor of $\varphi $ divides $(z+\mu)(z-\mu-1)$}
	\end{gathered}\right\}	
	$$ }
\end{theo}

We will also show that this set admits a natural stratification in terms of the Smith type of $\varphi $ (see Proposition~\ref{prop:SmithTypesStrat}). Moreover, we will see that there is a natural duality operation for polynomial Casimir representations and we will describe how the Smith types change under this duality.

In the last section we give an explicit construction of simple non-weight $\sl(2)$-modules of arbitrary finite rank as $\C[z]$-modules. More explictly, if $\mathbb A:= U/U(C-(2\mu+1)^2),$ where $U:=U(\sl(2))$ is the universal enveloping algebra of $\sl(2)$ and $C$ is the Casimir operator and $\mu\in\C$ with ${\mathfrak{Re}}(\mu)\geq -\frac12$, then one has:

\begin{theo}[see Theorem~\ref{thm:newsimpletorsionfree}]
Let $\alpha= X^n -   p(z)X^{n-1} - a_0 $ with  $p(z)\in\C[z]$, $a_0\in\C\setminus\{0\}$ and 
$\deg(p(z))\geq 1$. 

Then ${\mathbb A}/({\mathbb A}\alpha)$ is a simple non-weight $\sl(2)$-module of rank $n$ and semi-level $\mu$. 
\end{theo}

We finish this introduction by summarizing the contents of the paper. In Section \ref{sec:simplenonweight} we study the main properties of simple non-weight $\sl(2)$-modules. Polynomial Casimir representations are studied in detail in Section \ref{sec:polynomial}, whereas in Section \ref{sec:duality} we show that there is a duality for them and determine explicitly the correspondence of irreducible elements of the euclidean algebra $\mathbb A$ under this duality.  Section \ref{subsec:polynomialreprank1} is devoted to the complete determination of polynomial Casimir representations of rank 1. Finally, in Section \ref{sec:family} we carry out  the construction of a family of simple non-weight $\sl(2)$-modules of arbitrary finite rank.

\section{Simple non-weight $\sl(2)$-Modules}\label{sec:simplenonweight}

Recall that  $\sl(2)$ is a simple Lie algebra and a Chevalley basis for it consists of a basis $\{e,f,h\}$  satisfying the commutation relations~\eqref{eq:basis-sl2}. For the purposes of this paper, it will be more convenient to consider the basis $\{L_{-1}:=f, L_0:-\frac12h ,  L_1:=-e\}$.

We will find  a large class of simple non-weight modules, formed by the representations of $\sl(2)$ on finite rank modules over a polynomial ring in one variable and the class of representations on finite dimensional vector spaces over the field of rational functions. Our main reference for the theory of non-weight $\sl(2)$-modules will be  the papers of Bavula \cite{Bavula1, Bavula2} on the classification of simple $\sl(2)$-modules which extend previous work by Block \cite{Block}, see also the recent book by Mazorchuk \cite{Mazor}.


The Casimir operator can be expressed in the following equivalent ways
    \begin{equation}\label{align:casimir} 
    \begin{aligned}
    C &\,=\, (h+1)^2+4fe \,=\,  
    \\
    &\,=\, 
    4\big( (L_0-\frac{1}{2})^2-L_{-1}L_1 \big) 
    \,=\, 4\big( (L_0+\frac{1}{2})^2-L_1L_{-1}\big).
    \end{aligned}
    \end{equation}

\begin{defn} We say that an $\sl(2)$-module $V$ is a Casimir module if $C$ acts on $V$ by a constant.
\end{defn}

As a consequence of the Schur lemma, every simple $\sl(2)$-module is a Casimir module \cite[Thm. 4.7]{Mazor}. Taking into account the structure of the primitive spectrum of $\sl(2)$, it is customary to write the constant value that $C$ takes on a Casimir module $V$ as $(\lambda+1)^2$ for some $\lambda\in\C$ and in this case we say that $V$ is a $\lambda$-level Casimir module, or equivalently a $-(\lambda+2)$-level Casimir module. If one desires to associate a unique $\lambda$, then one may restrict to $\frak{Re}(\lambda)\geq -1$.


As examples of Casimir modules let us mention the dense modules as well as its submodules and quotients. In particular, the $n$-dimensional simple module $\mathbf{V}^{(n)}$ is a Casimir module of level  $(n-1)$, and the the Verma module $M(\lambda)$ and the anti-Verma module $\bar M(\lambda)$ are Casimir of level $\lambda$ and $\lambda-2$, respectively. Note that Casimir modules need not be simple and, thus, dense modules are also instances of Casimir modules.  

Every $\sl(2)$-module $V$ has a natural $\C[z]$-module structure, where $z$ acts on $V$ by $L_0$. That is, given a representation $\rho:\sl(2)\to \End_{\C}(V)$, then $z\cdot v:= \rho(L_0)(v)$, for every $v\in V$. Let us consider now the $\C$-algebra automorphism $\nabla\colon \C[z]\to \C[z]$ such that $\nabla(z)=z+1$.

\begin{defn}\label{defn:semilinear} Let $V$ be a $\C[z]$-module and $ k\in\Z$. We denote by $\End_k(V)$ the $\C[z]$-module of $\nabla^k$-semilinear endomorphisms of $V$; i.e., 
$$\End_k(V)\,:=\,
\big\{ \varphi \in \End_{\C}(V) \text{ s.t. }\varphi (z\cdot v)=\nabla^k(z)\cdot \varphi(v)=(z+k)\cdot \varphi(v)\big\} .$$ 
\end{defn}

\begin{rem}Notice that an $\sl(2)$-module structure on $V$ consists of a $\C[z]$-module structure and two $\C[z]$-semilinear endomorphisms 	
	$$\rho(L_{-1})\in \End_{-1}(V),\quad \rho(L_1)\in \End_{1}(V),$$ 
that satisfy $ [\rho(L_{-1}),\rho(L_1)]=-2z.$ In what follows we think of every $\sl(2)$-module in this way.  Given such a $\rho$ defined on a $\C[z]$-module $V$, we denote by $V_\rho$ the corresponding $\sl(2)$-module.
\end{rem}

In this way we  consider the category $\sl(2)\mathrm{-Mod}$ of $\sl(2)$-modules as a subcategory of $\C[z]\mathrm{-Mod}$. It follows that the inclusion functor $\sl(2)\mathrm{-Mod}\hookrightarrow \C[z]\mathrm{-Mod}$ is faithful and therefore for any pair $V, V'$ of $\sl(2)$-modules one has $$\Hom_{\sl(2)}(V,V')\subset \Hom_{\C[z]}(V,V').$$ In particular if $V, V'$ are isomorphic $\sl(2)$-modules then they are isomorphic considered as $\C[z]$-modules. Therefore an isomorphism class of $\sl(2)$-modules underlies an isomorphism class of $\C[z]$-modules. Hence it is natural to consider the space $\sl(2)\mathrm{-Mod}(V)$ of $\sl(2)$-module structures defined on a $\C[z]$-module $V$. Now one has that two representations $\rho, \rho'\in \sl(2)\mathrm{-Mod}(V)$ define isomorphic $\sl(2)$-modules, $V_\rho\simeq V_{\rho'}$, if and only if there exists $\varphi\in \GL_{\C[z]}(V)$ such that $\rho'(A)=\varphi\circ\rho(A)\circ\varphi^{-1}$ for every $A\in\sl(2)$.

Casimir modules of level $\lambda$ form a particular class of $\C[z]$-modules endowed with two $\C[z]$-semilinear endomorphisms as above fulfilling
    \begin{align*}
    \rho(L_{-1})\circ\rho(L_1)  &=\left(z-\frac{1}{2}\right)^2-\left(\frac{\lambda}{2}+\frac{1}{2}\right)^2, \\
    \rho(L_1)\circ\rho(L_{-1}) &= \left(z+\frac{1}{2}\right)^2-\left(\frac{\lambda}{2}+\frac{1}{2}\right)^2.
    \end{align*} 
The commutation relation, $ [\rho(L_{-1}),\rho(L_1)]=-2z$, follows automatically from these expressions. We  simplify the writing by putting $\mu=\frac{\lambda}{2}$ and refer to it as the semi-level of the Casimir module. We will assume that $\frak{Re}(\mu)\geq -\frac12$.


Defining the polynomial 
	$$\pi_\mu(z)=\left(z-\frac{1}{2}\right)^2-\left(\mu+\frac{1}{2}\right)^2=(z+\mu)(z-\mu-1)\in\C[z],$$
 we have 
\begin{equation}\label{eq:casimir-module}
\rho(L_{-1})\circ\rho(L_1) =\pi_\mu(z)\Id_V, \quad \rho(L_1)\circ\rho(L_{-1}) =\pi_\mu(z+1)\Id_V.
\end{equation}

\begin{defn} We say that an $\sl(2)$-module is torsion free (finitely generated) if it is torsion free (finitely generated) when considered as a $\C[z]$-module.
\end{defn}

One has the following dichotomy for simple modules.

\begin{thm}{\cite[Thm. 6.3]{Mazor}}\label{thm:dichotomy} A simple $\sl(2)$-module is either a weight module or a torsion free module.
\end{thm}

Casimir (simple) $\sl(2)$-modules of semi-level $\mu$ are exactly the (simple) modules over the generalized Weyl algebra  $$\mathbb A:= U/U(C-(2\mu+1)^2),$$ where $U:=U(\sl(2))$ is the universal enveloping algebra of $\sl(2)$, see for instance \cite[Chapter 6]{Mazor}. Moreover, $\mathbb A$ is a $\Z$-graded algebra $$\mathbb A=\bigoplus_{i\in\Z} \mathbb A_i,$$ with $\mathbb A_0=\C[z]$, $\mathbb A_{-i}=\mathbb A_0\cdot (L_{-1})^i$, $\mathbb A_{i}=\mathbb A_0\cdot (L_{1})^i$, for $i>0$.

The description given by Bavula \cite{Bavula1, Bavula2} of all simple $\sl(2)$-modules is based on the euclidean algebra $\mathbb B$ of skew Laurent polynomials over the field of rational fractions $\C(z)$ defined by the extension of the automorphism $\nabla$. Following the standard notation, we write   $\mathbb B=\C(z)[X,X^{-1};\nabla]$ whose product is determined  by the condition
	$$X^i\cdot \xi(z)=\nabla^i(\xi(z))\cdot X^i=\xi(z+i)\cdot X^i,$$
for every $\xi(z)\in\C(z)$ and every $i\in\Z$. This is an Euclidean algebra with respect to the length function
	$$\operatorname{length}_X(\sum_{i=m}^n a_i(z)X^i)\,:=\,  n-m\quad\text{ if }a_n(z), a_m(z)\neq 0. $$
and the following result is well known. 

\begin{prop}\label{prop:simpleB=B/Ba} The algebra $\mathbb B$ is both a left and a right principal ideal domain. Every simple $\mathbb B$-module is isomorphic to one of the form $\mathbb B/(\mathbb B\alpha)$ for some irreducible element $\alpha\in\mathbb B$.
\end{prop}

If one considers the multiplicative subset $S=\C[z]\setminus \{0\}$, then $\mathbb A$ embeds naturally in the localization $S^{-1}\mathbb A$ and there is a natural identification $\mathbb B\simeq S^{-1}\mathbb A$ such that $X=L_{1}$, $X^{-1}=\frac1{\pi_\mu(z)} L_{-1}$. In what follows we identify  $\mathbb A$ with its image inside $\mathbb B$. 


\begin{prop}[{\cite[Props. 3, 7]{Bavula2}}] \label{prop:Bavula1} Let $V$ be a simple torsion free $\sl(2)$-module of semi-level $\mu$. Then:
\begin{enumerate}\item $S^{-1}V$ is a simple $\mathbb B$-module that contains the $\mathbb A$-module $V$.
\item $V\simeq V_\alpha:= \mathbb A/(\mathbb A\cap\mathbb B\alpha)$ for some irreducible element $\alpha \in\mathbb B$, which we assume to belong to the subalgebra $\mathbb A^+:=\oplus_{i\geq 0}\mathbb A_i$.
\end{enumerate}
Moreover, for any irreducible $\alpha\in\mathbb A$, the sequence of $\mathbb A$-modules:
$$0\to E_\alpha :=(\mathbb A\cap \mathbb B\alpha)/(\mathbb A\alpha) \to \mathbb A/(\mathbb A\alpha)\to V_\alpha\to 0$$ is exact,  and $E_\alpha$ is a torsion $\sl(2)$-module of finite length.

\end{prop}

There is a natural functor \cite{Bavula1, Bavula2}, \cite[Sect. 6.3]{Mazor} of extension of scalars $$F=\C(z)\otimes_{\C[z]}(-)\colon \sl(2)_\mu\mathrm{-Mod}\to\mathbb B\mathrm{-Mod}.$$ In particular one has the following classification result.

\begin{thm}[{\cite[Thm. 6.24]{Mazor}}]\label{thm:isom-classes} The functor  $F$ induces a bijection, $\widehat F$, from the set of isomorphism classes of simple torsion free Casimir $\sl(2)$-modules of semi-level $\mu$ to the set of isomorphism classes of simple $\mathbb B$-modules.
\end{thm}


\begin{defn} 
Let $V$ be a $\sl(2)$-module. We say that $V$ is a polynomial (resp. rational) $\sl(2)$-module or that we have a polynomial (resp. rational) representation of $\sl(2)$ if there is an isomorphism of $\C[z]$-modules $V\simeq \C[z]^m$ (resp. $V\simeq \C(z)^m$). In this case, we say that $V$ has rank $m$. 
\end{defn}

\begin{thm}\label{thm:simpletorsionfreefingen-finiterankfree} 
Every simple torsion free finitely generated $\sl(2)$-module can be obtained as a polynomial  Casimir representation of $\sl(2)$. 
\end{thm}

For the proof, we need the following two Lemmas. 

\begin{lem}\label{lem:findim} Any simple $\mathbb B$-module is a finite dimensional vector space over the field of rational functions $\C(z)$.
\end{lem}

\begin{proof} Any simple $\mathbb B$-module is isomorphic to one of the form $\mathbb B/(\mathbb B\alpha),$ for some irreducible element $\alpha$. If necessary, after multiplying $\alpha$ by an invertible element we may assume that $\alpha\in \C[z][X]$. The claim is now obvious since $\dim_{\C(z)}(\mathbb B/(\mathbb B\alpha))=\operatorname{length}_X(\alpha)$.
\end{proof}


\begin{lem}\label{lem:simpletorsionfree-finiterank} Any simple torsion free $\sl(2)$-module is a finite rank torsion free $\C[z]$-module.
\end{lem}

\begin{proof} Given a simple torsion free $\sl(2)$-module $V$, the Schur lemma implies that it is a $\mu$-semi-level Casimir module for some $\mu\in \C$. By Theorem \ref{thm:isom-classes}  and Proposition~\ref{prop:simpleB=B/Ba}, there exists an irreducible element $\alpha\in\mathbb B$ such that 
$$\C(z)\otimes_{\C[z]}V\simeq \mathbb B/(\mathbb B\alpha),$$ as $\C(z)$-vector spaces. 
In particular, by Lemma \ref{lem:findim} one has that $\C(z)\otimes_{\C[z]}V$ is a finite dimensional $\C(z)$-vector space. Therefore, $V$ is a finite rank torsion free $\C[z]$-module.
\end{proof}

\begin{proof}[Proof of Thm. \ref{thm:simpletorsionfreefingen-finiterankfree}]
Let $V$ be a simple torsion free finitely generated $\sl(2)$-module. 
From Lemma \ref{lem:simpletorsionfree-finiterank} one knows that $V$ is a finite rank torsion free $\C[z]$-module. Since $\C[z]$ is a principal ideal domain, the conclusion follows. 
\end{proof}

\begin{rem}
Notice that every rational $\sl(2)$-module $V$ is necessarily torsion free. However, rational $\sl(2)$-modules are always not simple since they have uncountable dimension as $\C$-vector spaces, whereas simple  $\sl(2)$-modules, which can be realized as quotients of the universal enveloping algebra $U(\sl(2))$, have at most countable $\C$-dimension due to the Poincar\'e-Birkhoff-Witt theorem.
\end{rem}

\begin{rem} Given a polynomial $\sl(2)$-module $V$, by extension of scalars we get the rational $\sl(2)$-module $\C(z)\otimes_{\C[z]}V$. Moreover, there is a natural inclusion $V\hookrightarrow \C(z)\otimes_{\C[z]}V$ of $\sl(2)$-modules that shows explicitly that a rational module obtained by extension of scalars is not simple.  However in general it is not true that a rational $\sl(2)$-module $W$ can be obtained as the extension of scalars of a polynomial $\sl(2)$-module.
\end{rem}

We conclude with a criterion on polynomial representations.  For the proof, we need the following reformulation of \cite[Prop. 9]{Bavula2}: 

\begin{prop}\label{prop:Bavula2} 
Let $\alpha=a_0(z)+a_1(z) X+\cdots+a_n(z) X^n\in\mathbb A^+$, where $n\geq 0$, $a_i(z)\in\C[z]$, and let $F_\alpha$ be the $\C$-vector subspace of $\mathbb A$ consisting of the vectors of the form 	\[
	\sum_{i> 0}p_i(z) Y^{i}+
	\sum_{j\geq 0}r_j(z) X^{n+j},
	\]
where $Y=\pi_\mu(z)X^{-1}=L_{-1}$, and $\deg(p_i(z))<\deg(a_0(z))$, $\deg(r_j(z))<\deg(a_n(z))$. Then there is a direct sum decomposition as $\C$-vector spaces
	\begin{equation}\label{eq:directsumdecomp}
	\mathbb A=F_\alpha\oplus\mathbb A_{[n-1]}\oplus(\mathbb A\alpha),
	\end{equation}
where $\mathbb A_{[n-1]}=\mathbb A_0\oplus\cdots\oplus\mathbb A_{n-1}$.
\end{prop}

\begin{prop}\label{prop:submodule} 
With the same notations as in Proposition~\ref{prop:Bavula2}. The following conditions are equivalent:
\begin{enumerate}
\item $\mathbb A/\mathbb A\alpha$ is finitely generated as a $\C[z]$-module;
\item $a_0(z), a_n(z)$ belong to $\C^*$;
\item $\mathbb A/\mathbb A\alpha$ is a polynomial $\sl(2)$-module.
\end{enumerate}
\end{prop}

\begin{proof}
(1)$\implies$(2). Suppose that $F_{\alpha}\neq 0$. If $\mathbb A/\mathbb A\alpha$ is finitely generated as a $\C[z]$-module, then Proposition~\ref{prop:Bavula2} allows us to obtain a system of generators consisting of $1,X\ldots, X^{n-1}$ of $\mathbb A_{[n-1]}$ together with finitely many elements of $F_{\alpha}$, say $\{f_1,\ldots, f_r\}$ (where $r\geq 1$ since $F_{\alpha}\neq 0$). Let  us fix $k$ bigger than $n$, $\deg_Y(f_i)$ and $\deg_X(f_i)$ for all $i$. It follows that the set $	\{ Y^k, Y^{k-1}, \ldots, Y, 1, X, \ldots, X^k\}$ generates $\mathbb A/\mathbb A\alpha$ as a $\C[z]$-module. This is equivalent to say that $\{ Y^k, Y^{k-1}, \ldots, Y, 1, X, \ldots, X^k\}$ together with $\{Y^i\alpha, X^i\alpha\vert i\in {\mathbb Z}\} $ generate $\mathbb A$ as a $\C[z]$-module. 

In particular, $Y^{k+1}$ would be a linear combination of these elements; that is, there exists a relation in ${\mathbb A}$ of the following type
	\[
	Y^{k+1}= \sum_{i>0} \lambda_i(z) Y^i\alpha 
	+ \sum_{i=1}^k \mu_i(z) Y^i 
	+ \sum_{i=0}^k \nu_i(z) X^i 
	+ \sum_{i>0} \xi_i(z) X^i\alpha.
	\]
where it can be assumed that $\deg(\mu_i(z))<\deg(a_0(z))$ for all $i$ and that $\deg(\nu_i(z))<\deg(a_n(z))$ for all $i\geq n$. Having in mind the decomposition~\eqref{eq:directsumdecomp} as $\C$-vector spaces and that $Y^{k+1}\in F_{\alpha}$, we obtain a contradiction unless $a_0(z)$ is constant. Arguing similarly with $X^{k+1}$, we obtain that $\alpha_n(z)\in \C^*$. Hence $F_{\alpha}=0$.
%
%
%
(2)$\implies$(3). By Proposition~\ref{prop:Bavula2},  the free $\C[z]$-module of finite rank $\mathbb A_{[n-1]}$ is naturally a $\C[z]$-submodule of the quotient $\mathbb A/(\mathbb A\alpha)$. Define the $\C[z]$-module $Q$ by the exact sequence 
	\[ 
	0\to \mathbb A_{[n-1]}\to \mathbb A/\mathbb A\alpha \to Q \to 0 
	\]
and observe that there is a surjection $F_{\alpha}\to Q$. Since, by the hypothesis, $\deg(a_0(z))=\deg(a_n(z))=0$, it follows that $F_{\alpha}=0$ and, consequently, $Q=0$. 


(3)$\implies$(1). Obvious. 
\end{proof}


\section{Polynomial Casimir representations}\label{sec:polynomial}

Let $V$ be a free $\C[z]$-module of rank $n$. After choosing a base of $V$ we can identify it with $\C[z]^n$ and, if no confusion arises, we use $V$ and $\C[z]^n$ interchangeably. The automorphism $\nabla\colon \C[z]\to\C[z]$ extends in a natural way to a $\C[z]$-semilinear automorphism of  $V\simeq\C[z]^n$ and we continue to denote its extension by the same letter.

Following the Definition~\ref{defn:semilinear} of the previous section, we denote by $\End_k(V)$ the $\C[z]$-module of $\nabla^k$-semilinear endomorphisms of $V$. In the same way
	\begin{equation}\label{eq:gradedalgebrasemilinear}
	\End_\bullet(V)\,:=\, \bigoplus_{k\in\Z}\End_k(V) \,\subseteq\, \End_{\C}(V),
	\end{equation}
is a $\Z$-graded $\C$-algebra which will be called the $\Z$-graded algebra of semilinear endomorphisms of the $\C[z]$-module $V$ with respect to the automorphism $\nabla$.

\begin{prop}\label{prop:endk=nablak}
It holds that
$$\End_k(V)\,=\, \End_{\C[z]}(V) \cdot \nabla^k= \, \nabla^k \cdot \End_{\C[z]}(V). $$
\end{prop}

\begin{proof}
Having in mind that $\End_0(V)=\End_{\C[z]}(V)$, $\nabla^k\in\End_k(V)$ for all $k\in\Z$ and that $\nabla$ is an automorphism, the claim follows. 
\end{proof}

We denote by $\sl(2)_\mu\mathrm{-Mod}(V)$ the space of polynomial Casimir representations of semi-level $\mu\in\C$ defined on a $\C[z]$-module $V$.

In order to define on $V$ a Casimir representation of semi-level $\mu\in\C$,  we must give two semilinear endomorphisms $\rho(L_{-1})\in\End_{-1}(V)$, $\rho(L_1)\in\End_{1}(V)$ that must satisfy equation (\ref{eq:casimir-module}).

Taking into account Proposition \ref{prop:endk=nablak}, equation \eqref{eq:casimir-module}  and the fact that $\rho(L_k)\in \End_{k}(V)$ for $k=-1,0,1$, we immediately have the following result.
 
\begin{prop} 
Let $V$ be a free $\C[z]$-module of finite rank. The group $\operatorname{GL}_{\C[z]}(V)\times \operatorname{GL}_{\C[z]}(V) $ acts on  $\sl(2)_\mu\mathrm{-Mod}(V)$ by $((g,h),\rho)\mapsto \rho^{(g,h)}$ where
	$$
	\rho^{(g,h)}(L_{-1}) \,:=\, h\circ \rho(L_{-1}) \circ g^{-1}\quad ,  \quad
	\rho^{(g,h)}(L_{1}) \,:=\,  g\circ \rho(L_{1}) \circ h^{-1}.
	$$

\end{prop}

In particular, $\rho,\rho'\in\sl(2)_\mu\mathrm{-Mod}(V)$ are equivalent $\sl(2)$-re\-pre\-sen\-ta\-tions if and only if they are related by the diagonal action, that is  $\rho'=\rho^{(g,g)}$ for some $g\in\operatorname{GL}_{\C[z]}(V)$. This motivates the following 

\begin{defn} We say that two representations $\rho,\rho'\in\sl(2)_\mu\mathrm{-Mod}(V)$ are weakly equivalent  if there exists $(g,h)\in \GL_{\C[z]}(V)\times \GL_{\C[z]}(V)$ such that $\rho_2=\rho_1^{(g,h)}$.  
\end{defn}

\begin{prop}\label{prop:endomorphisms-polynomial-casimir-rep} Let $\rho\colon\sl(2)\to\End(V)$ be a polynomial Casimir representation of semi-level $\mu$. Then one has 
$$\End_{\sl(2)}(V_\rho)=\{\phi\in\End_{\C[z]}(W)\colon [\phi,\rho(L_1)]=0\}.$$
\end{prop}

\begin{proof} It is obvious that elements of $ \End_{\sl(2)}(V_\rho)$ satisfy the condition of the statement. On the other hand, if $\phi\in\End_{\C[z]}(V)$ verifies $[\phi,\rho(L_1)]=0$, then expanding the commutator of $\phi$ with the second equation in \eqref{eq:casimir-module} and taking into account that $V$ is torsion free, we also get $[\phi,\rho(L_{-1})]=0$, proving the claim.
\end{proof}


\begin{thm}\label{thm:polynomialmuCasimir}
Let $V$ be a rank $n$ free $\C[z]$-module. There is an identification 
	$$
	\left\{
	\begin{gathered}
	\varphi\in \End_1(V) \text{ s.t. the $n$-th}
	\\
	\text{invariant  factor of $\varphi $ divides $\pi_{\mu}(z+1)$}
	\end{gathered}\right\}	
	\longleftrightarrow \sl(2)_\mu\mathrm{-Mod}(V),$$ 
that sends $\varphi\in \End_{1}(V)$ to the $\sl(2)$-representation 
	$$
	\rho(L_{-1})\, :=\, \pi_\mu(z) \varphi^{-1}  
		\quad , \quad 
		\quad 	\rho(L_1)\, :=\, \varphi\,.  	$$
\end{thm}

The proof requires some Lemmas. 

\begin{lem}\label{lem:A1}
Let $\rho\in \sl(2)_\mu\mathrm{-Mod}(V)$. There exist $A_{-1},A_1\in\End_{\C[z]}(V)$ such that 
	\begin{equation}\label{eq:rho=Anabla}
	\rho(L_{-1})=A_{-1}\circ \nabla^{-1},\quad \rho(L_{1})=A_{1}\circ \nabla.
	\end{equation}
\end{lem}

\begin{proof}
Given $\rho$, a Casimir representation of semi-level $\mu\in\C$, one has that  $\rho(L_{-1})\in\End_{-1}(V)$ and $\rho(L_1)\in\End_{1}(V)$. Bearing in mind Proposition~\ref{prop:endk=nablak} 
the claim follows. 
\end{proof}

Since we are assuming an identification $V\simeq \C[z]^n$, every $A\in \End_{\C[z]}(V)$ has an associated matrix $A(z) \in M_{n\times n}(\C[z])$. 

\begin{lem}\label{lem:A1(z)}
Two endomorphisms $A_{-1},A_1\in\End_{\C[z]}(V)$ define a Ca\-si\-mir representation on $V$ of semi-level $\mu\in\C$ (via equation \eqref{eq:rho=Anabla}) if and only if their associated matrices $A_{-1}(z), A_1(z)$ w.r.t. some basis satisfy
\begin{equation}\label{eq:matrix-cond} A_{-1}(z)A_1(z-1)  =\pi_\mu(z)\Id, \quad A_1(z)A_{-1}(z+1) =\pi_\mu(z+1)\Id.
\end{equation}
\end{lem}

\begin{proof}
In order to define on $V$ a Casimir representation of semi-level $\mu\in\C$,  we must give  $\rho(L_{-1})\in\End_{-1}(V)$ and $\rho(L_1)\in\End_{1}(V)$ satisfying equation~\eqref{eq:casimir-module}. 
It follows that $\rho(L_{-1})$ and $\rho(L_1)$ are injective maps. 

Since $\nabla$ is an automorphism, it acts by conjugation on $\End_{\C[z]}(V)$; that is, given $A\in\End_{\C[z]}(V)$ one defines $\nabla(A)=\nabla\circ A\circ\nabla^{-1}$ which is a ring automorphism of $\End_{\C[z]}(V)$.  If we identify $A$ with a matrix $A(z)$ via the isomorphism $\End_{\C[z]}(V)\simeq M_{n\times n}(\C[z])$, then $\nabla(A)$ gets identified with the matrix $A(z+1)$.

Now, the specification in \eqref{eq:rho=Anabla} defines a Casimir representation $\rho$ if and only if  $A_{-1},A_1\in\End_{\C[z]}(V)$  fulfill 
$$
A_{-1}\circ\nabla^{-1}(A_1)  =\pi_\mu(z)\Id_V, \quad A_1\circ\nabla(A_{-1}) =\pi_\mu(z+1)\Id_V,$$ 
which,  in terms of the associated matrices, is equivalent to the statement. 
\end{proof}

Let us recall that every matrix $M\in M_{n\times n}(\C[z])$ has a Smith normal form \cite[Thm. 7.7.1]{Hazewinkel}, \cite[Thm. II.9]{Newman}. That is, there exist matrices $U, V\in\GL(n,\C[z])$ and a unique diagonal matrix $S\in M_{n\times n}(\C[z])$, the Smith normal form of $M$, which we also denote by $S(M)$, such that: \begin{enumerate}
\item $M=USV$. 
\item $S$ is a diagonal matrix whose entries are monic polynomials, $S=\operatorname{Diag}\{s_1(z),\ldots,s_n(z)\}$.
\item $s_{i}(z)$ is the $i$-th invariant factor of $M$.
\end{enumerate}

\begin{rem}\label{rem:weakly-equiv-smith-form} The Smith normal forms of two weakly equivalent matrices $M\equiv M'$ are the same. That is, given $M, M'\in M_{n\times n}(\C[z])$ such that there exist $U', V'\in\GL(n,\C[z])$ verifying $M'=U'MV'$, then $S(M)=S(M')$.
\end{rem}

%

\begin{lem}\label{lem:invariantfactor}
Let $\varphi=A_k\circ \nabla^k \in \End_k(V)$ and  let $A_k(z)\in M_{n\times n}(\C[z])$ be the matrix associated to $A_k\in\End_{\C[z]}(V)$. 

Then, the Smith normal form of $A_k(z)$ does not depend on the choice of the isomorphism $V\simeq \C[z]^n$. 
\end{lem}

\begin{proof} 
Let $\psi$ be an isomorphism  $V\overset{\sim}\to \C[z]^n$. Any other isomorphism can be written as $g\circ\psi$ with $g\in \GL(n,\C[z])$. By definition the matrix associated to $A_k$ w.r.t. the isomorphism $\psi$ is $A_k(z)=\psi\circ A_k\circ \psi^{-1}$ and the one associated w.r.t. $g\circ\psi$ is $A'_k(z)=(g\circ\psi)\circ A_k\circ (g\circ\psi)^{-1}$ .  Therefore one has
	$$A'_k(z)=(g\circ\psi)\circ A_k\circ (g\circ\psi)^{-1}=g\circ A_k(z)\circ g^{-1},  $$
and the conclusion follows by the previous Remark since $g\in \GL(n,\C[z])$.
\end{proof}

Bearing in mind the previous Lemma, we define the invariant factors of $\varphi\in\End_k(V)$ as the invariant factors of $A(z)$ where $\varphi$ corresponds to $A(z)\nabla^k$ via any isomorphism $V\simeq \C[z]^n$. 

\begin{lem}\label{lem:smith-determination} 
Let $A_{-1}(z),A_1(z)\in M_{n\times n}(\C[z])$ satisfy \eqref{eq:matrix-cond}. Let $S(z)$ denote the Smith normal form of $A_1(z)$. 

Then, there exists a diagonal matrix $T(z)\in M_{n\times n}(\C[z])$ such that:
\begin{enumerate}\item $S(z)T(z+1)=\pi_\mu(z+1)\Id$,
\item $A_{-1}(z)=V(z-1)^{-1}T(z)U(z-1)^{-1}$.
\end{enumerate}
\end{lem}

\begin{proof} 
Let $U(z), V(z)\in\GL(n,\C[z])$ be such that  $A_1(z)=U(z)S(z)V(z)$. Plugging $A_1(z)=U(z)S(z)V(z)$ into the second equation of (\ref{eq:matrix-cond}) we get $U(z)S(z)V(z)A_{-1}(z+1) =\pi_\mu(z+1)\Id.$ This implies $$S(z)V(z)A_{-1}(z+1)U(z) =\pi_\mu(z+1)\Id.$$ Since $S(z)$ is diagonal, it follows that $T(z+1)=V(z)A_{-1}(z+1)U(z)$ is also a diagonal matriz and the claim follows.
\end{proof}

\begin{rem}
If one has 
$$S(z)=\diag\{s_1(z),\ldots,s_n(z)\},\quad T(z)=\diag\{t_1(z),\ldots,t_n(z)\},$$ where $s_i(z), t_i(z)\in\C[z]$, then it follows that $$s_i(z)\cdot t_i(z+1)=\pi_\mu(z+1).$$ 
\end{rem}

\begin{proof}[Proof of Thm~\ref{thm:polynomialmuCasimir}]
First, note that the l.h.s. of the statement is well defined due to Lemma~\ref{lem:invariantfactor}. 

Let $\rho\in \sl(2)_\mu\mathrm{-Mod}(V)$ be given and let $A_{-1},A_1$ be the endomorphisms obtained as in Lemma~\ref{lem:A1}. The $n$-th invariant factor of $A_1$ is the last diagonal entry of $S(z)$ which, by Lemma~\ref{lem:smith-determination}, divides $\pi_{\mu}(z+1)$. It is straightforward to see that $\varphi:=A_1$ is the desired endomorphism. 

For the converse, it suffices to show that 
	$$ \pi_\mu(z)\,  \varphi^{-1} \,\in\, \End_{\C[z]}(V)$$
 but this follows from the fact that the $n$-th invariant factor of $\varphi$ divides $\pi_{\mu}(z+1)$. This agrees with the second item of Lemma~\ref{lem:smith-determination}.
\end{proof}

%
%

\begin{defn} Given any polynomial Casimir representation $\rho\in \sl(2)\mathrm{-Mod}(V,\mu),$ we define its Smith type $S(\rho)\in M_{n\times n}(\C[z])$ as the Smith type of $\rho(L_1)$, i.e.  $S(\rho):=S(\rho(L_1))$.
\end{defn}

By Remark \ref{rem:weakly-equiv-smith-form} we immediately get the following result.

\begin{prop} Polynomial Casimir representations that are weakly equivalent have the same Smith type.
\end{prop}

Since equivalent representations are weakly equivalent, we conclude that the Smith type of any two equivalent polynomial Casimir representations is the same.  This shows that the space of equivalence classes of polynomial Casimir representations is stratified according to the Smith type. More precisely,  recalling that $\pi_\mu(z+1)=\alpha_\mu(z+1)\cdot\beta_\mu(z+1)$, with $\alpha_\mu(z+1)=z+\mu+1$, $\beta_\mu(z+1)=z-\mu$, one has:

\begin{prop}\label{prop:SmithTypesStrat} The possible Smith types for the polynomial Casimir representations of semi level $\mu$ on a free $\C[z]$-module of rank $n$ are{\small  \begin{align*}
S_+(i,j,k) &=\diag(1,\stackrel{^{i)}}{\ldots},1,\alpha_\mu(z+1),\stackrel{^{j)}}{\ldots},\alpha_\mu(z+1),\pi_\mu(z+1),\stackrel{^{k)}}{\ldots},\pi_\mu(z+1)),\\S_0(l,m) &=\diag(1,\stackrel{^{l)}}{\ldots},1,\pi_\mu(z+1),\stackrel{^{m)}}{\ldots},\pi_\mu(z+1)),\\ S_-(i,j,k) &=\diag(1,\stackrel{^{i)}}{\ldots},1,\beta_\mu(z+1),\stackrel{^{j)}}{\ldots},\beta_\mu(z+1),\pi_\mu(z+1),\stackrel{^{k)}}{\ldots},\pi_\mu(z+1)),\end{align*}}for some non negative integers $i,j,k, l, m$ such that $i+j+k=n$, $j>0$, $l+m=n$.
\end{prop}


We denote $\mathcal S(n,\mu)=\mathcal S_-(n,\mu)\coprod\mathcal S_0(n,\mu)\coprod\mathcal S_+(n,\mu)$, where {\tiny$$\mathcal S_-(n,\mu)=\coprod_{\superpuesto{i+j+k=n,}{i\geq 0, j>0,k\geq 0}}S_-(i,j,k),\ \mathcal S_0(n,\mu)=\coprod_{\superpuesto{l+m=n}{l\geq0, m\geq 0}}S_0(l,m),\ \mathcal S_+(n,\mu)=\coprod_{\superpuesto{i+j+k=n,}{i\geq 0, j>0,k\geq 0}}S_-(i,j,k).$$} 
\begin{rem} A simple computation shows that $\#S_0(n,\mu)=n+1$, $\# S_-(n,\mu)=\#S_+(n,\mu)=\frac{n(n+1)}{2}$  and thus  $\mathcal S(n,\mu)$ has $(n+1)^2$ elements.
\end{rem}

\begin{prop} Let $V$ be a free $\C[z]$-module of finite rank $n$. There is a surjective map $$\GL(n,\C[z])\times \mathcal S(n,\mu)\times \GL(n,\C[z])\xrightarrow{\Phi} \sl(2)_\mu\mathrm{-Mod}(V)$$ defined by mapping $(U(z),S(z),V(z))$ to the polynomial representation defined by $\rho(L_1)=U(z)S(z)V(z)\circ\nabla$ w.r.t. a basis of $V$.
\end{prop}

\begin{rem}\label{rem:number-smith-types}
Notice that the map $\Phi$ is not injective. If one considers the equivalence of representations, $\rho\sim\rho'$ if there exists $g\in\GL_{\C[z]}(V)$ such that $\rho'=\rho^{(g,g)}$, then it follows that every equivalence class $[\rho]$ admits a unique representative of the form either $\rho(L_1)=U(z)S(z)\circ\nabla$ or $\rho(L_1)=S(z)V(z)\circ\nabla$.
\end{rem}


\section{Duality of polynomial Casimir representations}\label{sec:duality}

In this section we show that polynomial Casimir representations admit a natural duality that is compatible with irreducible representations. We also study its behavior with respect to the Smith type.

\begin{defn} Let $V$ be a free $\C[z]$-module of finite rank and let $V^*=\Hom_{\C[z]}(V,\C[z])$ be its dual $\C[z]$-module. For any representation $\rho\in\sl(2)_\mu\mathrm{-Mod}(V)$ and any $\omega\in V^*$, we define 
	$$
	\rho^{\vee}(L_{-1})(\omega) \,:=\nabla^{-1}\circ\omega\circ\rho(L_1)\quad ,  \quad
	\rho^{\vee}(L_{1})(\omega) \,:=\,  \nabla\circ\omega\circ\rho(L_{-1}).$$
\end{defn}

\begin{prop} For any representation $\rho\in\sl(2)_\mu\mathrm{-Mod}(V)$ one has that $\rho^\vee$ is a representation that belongs to $\sl(2)_\mu\mathrm{-Mod}(V^*)$.
\end{prop}

\begin{proof} Since semilinear endomorphisms form a $\Z$-graded algebra, it is easy to check that for any $\omega\in V^*$ one has $\rho^{\vee}(L_{-1})(\omega), \rho^{\vee}(L_{1})(\omega)\in V^*$. On the other hand  \begin{align*}\rho^{\vee}(L_{-1})(\rho^{\vee}(L_{1})(\omega))=\nabla^{-1}\circ \rho^{\vee}(L_{1})(\omega)\circ \rho(L_{1})=\\=\nabla^{-1}\circ\nabla\circ\omega\circ\rho(L_{-1})\circ\rho(L_{1})=\omega(\pi_\mu(z)\Id_V)=\pi_\mu(z)\omega.
\end{align*} In a similar way one sees that $\rho^{\vee}(L_{1})\circ \rho^{\vee}(L_{-1})=\pi_\mu(z+1)\Id_{V^*}$.
\end{proof}

\begin{defn} Given a representation $\rho\in\sl(2)_\mu\mathrm{-Mod}(V)$ we say that $\rho^\vee\in \sl(2)_\mu\mathrm{-Mod}(V^*)$ is its dual representation.  
\end{defn}

\begin{thm}\label{thm:irre-dual} A representation $\rho\in\sl(2)_\mu\mathrm{-Mod}(V)$ is irreducible if and only if its dual representation $\rho^\vee\in \sl(2)_\mu\mathrm{-Mod}(V^*)$ is irreducible.
\end{thm}

\begin{proof}  Let us suppose that $\rho$ is irreducible and let $W\subset V^*$ be a $\rho^\vee$-invariant subspace. In particular, $W$ is a $\C[z]$-submodule and therefore we may consider its annihilator $$W^\circ=\{v\in V\colon \omega(v)=0,\quad\forall\ \omega\in W\},$$ that is a $\C[z]$-submodule of $V$. One easily checks that $W^\circ$ is invariant under $\rho(L_1)$ and $\rho(L_1)$ and therefore it is an $\sl(2)$-submodule of $V$. Since $V$ is irreducible one must have that $W^\circ$ is either $\{0\}$ or $V$, and therefore since $V^*$ is a free $\C[z]$-module of finite type this implies that $W$ is either $\{0\}$ or $V^*$. The other implication is proved in a similar way.
\end{proof}

\begin{prop}
Let   $\alpha=\sum_{i=0}^n \alpha_i(z) X^i \in {\mathbb B}$ be an irreducible element such that $V:= {\mathbb A}/({\mathbb A}\cap{\mathbb B}\alpha)$ is a polynomial $\sl(2)$-representation and $\alpha_0(z)\cdot\alpha_n(z)\neq 0$.  If $V^*$ denotes the dual representation of $V$, then it holds that $V^*$  is a simple torsion free $\sl(2)$-module and there exists an isomorphism 
	\[ 
	V^*\,\simeq \, 
	{\mathbb A}/({\mathbb A}\cap{\mathbb B}\alpha^*),
	\]
where $\alpha^*$ is the irreducible element given by	
	\[
	\alpha^*\,=\, \sum_{i=0}^n \Big( 
	(\pi_{\mu}(z+1) \nabla)^{2i-n} \alpha_{n-i}(z) \Big) X^i.
	\]	
\end{prop}

\begin{proof}
Theorem~\ref{thm:irre-dual} implies that if $(V,\rho)$, i.e.  $\rho\in\sl(2)_\mu\mathrm{-Mod}(V)$, is a simple torsion free $\sl(2)$-module, then $(V^*, \rho^\vee)$ is also a simple torsion free $\sl(2)$-module. Bearing in mind Proposition \ref{prop:Bavula1}  and Theorem~\ref{thm:isom-classes}, it follows that given $V= {\mathbb A}/({\mathbb A}\cap{\mathbb B}\alpha)$ for some irreducible element $\alpha\in {\mathbb B}$, there must exist another irreducible element $\alpha^*\in {\mathbb B}$ such that the dual representation $V^*$ is isomorphic to  ${\mathbb A}/({\mathbb A}\cap{\mathbb B}\alpha^*)$. It remains to compute $\alpha^*$ explicitly. 

Consider the canonical $\C[z]$-bilinear pairing $V^*\times V\to \C[z]$. Tensoring it by $\C(z)$, recalling \S\ref{sec:simplenonweight} and composing with ${\mathbb B}\to {\mathbb B}/{\mathbb B} \alpha$, we obtain a $\C(z)$-bilinear map 
	\[
	\{\,,\, \}\,\colon\, 
	{\mathbb B}\,\times\, {\mathbb B}
	\, \longrightarrow \,
	{\mathbb B}/{\mathbb B} \alpha^*\,\times\, {\mathbb B}/{\mathbb B} \alpha \, 
	\longrightarrow \, \C(z).
	\]
Since $V$ is a polynomial $\sl(2)$-representation, it holds that $\operatorname{length}(\alpha^*)=\dim_{\C(z)} {\mathbb B}/{\mathbb B} \alpha^* = \dim_{\C(z)} {\mathbb B}/{\mathbb B} \alpha =\operatorname{length}(\alpha)=n$. Let us define  $a_{ij}(z):=\{X^i,X^j\}$. The compatibility of this map  w.r.t. the $\sl(2)$ action yields
	\[
	\begin{aligned}
	\big(\rho^\vee(L_1)(X^i)\big)(X^j) &\,=\, \{X\cdot X^i, X^j\}	\,=\, a_{i+1,j}(z)
	\\
	\big(\rho^\vee(L_1)(X^i)\big)(X^j) &\,=\, 
	\nabla \{ X^i, \rho(L_{-1}) X^j\}	\,=  \,\nabla \{ X^i, \pi_{\mu}(z) X^{-1}\cdot X^j\}	\,=
	\\ 
	& \,=\, \pi_{\mu}(z+1) a_{i,j-1}(z+1)
	\end{aligned}
	\]
that is, $\pi_{\mu}(z+1) a_{i,j-1}(z+1) =a_{i+1,j}(z)$. 
	
We may restrict the pairing $\{\,,\,\}$ to the subspace $<1,\ldots, X^n>\subset {\mathbb B}$ and, thus, we obtain
	\[
	\{\,,\, \}\,\colon\,<1,\ldots, X^n>\,\times \, <1,\ldots, X^n> 
	\,\longrightarrow\, {\mathbb B}/{\mathbb B} \alpha^*\,\times\, {\mathbb B}/{\mathbb B} \alpha \, 
	\longrightarrow \, \C(z)
	\]
and $A(z)=(a_{ij}(z)) $  is the $n\times n$-matrix associated to this bilinear map. 

Observe that $\alpha=\sum_{i=0}^n \alpha_i(z) X^i \in <1,\ldots, X^n>$ satisfies  $\{\beta,\alpha\}=0$ for all $\beta$ and it is characterized (up to an invertible) by this property. Similarly, $\alpha^*=\sum_{i=0}^n \alpha^*_i(z) X^i \in <1,\ldots, X^n>$ fulfills $\{\alpha^*,\beta\}=0$ for all $\beta$ and it is characterized (up to an invertible) by this property. 

Then, the fact that $\{\, , \alpha\}=0$ is expressed as $A(z){\small \begin{pmatrix} \alpha_0(z) \\ \vdots \\ \alpha_n(z) \end{pmatrix}}=0$; or, what is tantamount, 
	\begin{equation}\label{eq:alpha}
	B(z)\cdot \left( {\mathfrak D}^{-1} \begin{pmatrix} \alpha_0(z) \\ \vdots \\ \alpha_n(z) \end{pmatrix}\right)
	\,=\, 0
	\end{equation}
where ${\mathfrak D}$ is the operator-valued diagonal matrix
	\[
	{\mathfrak D} \,:=\, 
	\begin{pmatrix} 1 \\ & \pi_{\mu}(z+1)\nabla \\ & & \ddots \\ & & & (\pi_{\mu}(z+1)\nabla)^n \end{pmatrix}
	\]
and $B(z):= {\mathfrak D}^{-1} A(z)$. 

Analogously $\{\alpha^*,\,\}=0$ is written as ${\small\begin{pmatrix} \alpha_0^*(z) & \dots & \alpha_n^*(z) \end{pmatrix}}A(z)=0$; or, equivalently,
	\begin{equation}\label{eq:alpha*}
	\left( {\mathfrak D}^{-1}  \begin{pmatrix} \alpha_0^*(z) \\ \vdots \\ \alpha_n^*(z) \end{pmatrix}
	\right)^t \cdot B(z) \,=\, 0,
	\end{equation}
where the superscript $t$ denotes the transpose.

Our task consists of computing the solution of \eqref{eq:alpha*} in terms of the data \eqref{eq:alpha}. Recalling that $\pi_{\mu}(z+1) a_{i,j-1}(z+1) =a_{i+1,j}(z)$, one observes that $b_{ij}(z)$ depends only on $i-j$; that is, if $i-j=k-l$, then $b_{ij}(z)=b_{kl}(z)$. Using this fact, one concludes that the $(n-i+1)$-th component of 
 $ {\mathfrak D}^{-1} {\small \begin{pmatrix} \alpha_0(z) \\ \vdots \\ \alpha_n(z) \end{pmatrix})} $  coincides with the $i+1$-th component of  $ {\mathfrak D}^{-1} {\small \begin{pmatrix} \alpha_0^*(z) \\ \vdots \\ \alpha_n^*(z) \end{pmatrix})}$. Summing up
	\[
	\big(\pi_{\mu}(z+1) \nabla\big)^{-(n-i)} \alpha_{n-i}(z)\, =\, \big(\pi_{\mu}(z+1) \nabla\big)^{-i} \alpha_{i}^*(z)
	\]
and, thus, $\alpha_{i}^*(z)= \big(\pi_{\mu}(z+1) \nabla\big)^{2i-n} \alpha_{n-i}(z)$. 
\end{proof}

An easy computation proves the following result.

\begin{prop}  Let $\rho\in\sl(2)_\mu\mathrm{-Mod}(V)$ be a representation on a free $\C[z]$-module of rank $m$ and let $\rho^\vee\in\sl(2)_\mu\mathrm{-Mod}(V^*)$ be its dual representation. One has
\begin{enumerate}
\item If $\rho$ has Smith type $S_+(i,j,k)$ then $\rho^\vee$ has Smith type  $S_-(k,j,i)$.
\item  If $\rho$ has Smith type $S_0(l,m)$ then $\rho^\vee$ has Smith type  $S_0(m,l)$.
\item If $\rho$ has Smith type $S_-(i,j,k)$ then $\rho^\vee$ has Smith type  $S_+(k,j,i)$.
\end{enumerate}
\end{prop}

\section{Polynomial representations of rank one}\label{subsec:polynomialreprank1}

\begin{prop}\label{prop:rank-one-polynomial-are-casimir} Every polynomial  $\sl(2)$-representation on a  free $\C[z]$-module of rank one is a Casimir representation.
\end{prop}

\begin{proof} Let $\rho\colon \sl(2)\to \End(V)$ be an $\sl(2)$-representation on a free $\C[z]$- module $V$ of rank one. Since $V\simeq \C[z]$ we have $$\rho(L_{-1})=A_{-1}(z)\, \nabla^{-1},\quad \rho(L_{1})=A_{1}(z)\, \nabla,$$ where $A_{-1}(z), A_1(z)\in\C[z]$. Since we have the commutation relation $[\rho(L_{-1}),\rho(L_1)]=-2\, z$, these polynomials must satisfy \begin{equation}\label{eq:onedim-commut}
A_{-1}(z)A_1(z-1)-A_1(z)A_{-1}(z+1)=-2\, z.
\end{equation}
Define now $B(z):=A_{-1}(z)A_1(z-1)$, then (\ref{eq:onedim-commut}) is equivalent to the difference equation \begin{equation}\label{eq:dif-eq-commut} B(z+1)-B(z)=2\,z,\end{equation} that is $(\Delta B)(z)=2z$. Applying $\Delta^2$ to this equation we get $$\Delta^3 B=0.$$ Therefore $B(z)$ is a polynomial of degree 2 that has to satisfy (\ref{eq:dif-eq-commut}). After some computations we get $B(z)=(z-\frac1{2})^2+ \nu$ for some constant $\nu\in \C$. According to (\ref{align:casimir}) the Casimir operator is $C =4\,[\, (z-\frac{1}{2})^2-L_{-1}L_1 \,]$. On the other hand  one has $L_{-1}L_1 =B(z)$ and hence we get $C=-4\nu$, proving the claim for polynomial representations.
\end{proof}

\begin{cor}\label{cor:endomorphisms-rank-one} For any polynomial representation $\rho$ on a free $\C[z]$-module $V$ of rank one one has $\End_{\sl(2)}(V_\rho)=\C\cdot\Id_V$.
\end{cor}

\begin{proof} Let us suppose that $\rho$ is a polynomial representation. By Proposition \ref{prop:rank-one-polynomial-are-casimir} we know that $\rho$ is a Casimir representation. Therefore,  if we take into account Proposition \ref{prop:endomorphisms-polynomial-casimir-rep},  it follows that $\phi\in\End_\C(V)$ belongs to $\End_{\sl(2)}(V_\rho)$ if and only if $\phi\in\End_{\C[z]}(V)$ and $[\phi,\rho(L_1)]=0$. The later condition is equivalent to $$B(z)A_1(z)-A_1(z) B(z+1)=0,$$ where $B(z)$, $A_1(z)$ are the polynomials representing $\phi$, $\rho(L_1)$, respectively, on a basis. Since $A_1(z)\neq0$ the condition above is equivalent to $\Delta B=0$. Hence $B\in\C$.
\end{proof}

\begin{thm}\label{thm:casimir-reps-one-dim} The Casimir representations of semi-level $\mu$ on a rank one free $\C[z]$-module $V$,  are:
\begin{enumerate}
\item[I)] $\rho(L_{-1})=\frac1{\gamma}\cdot\nabla^{-1},\quad\quad\quad\ \rho(L_0)=z,\quad  \rho(L_{1})={\gamma}\cdot\pi_\mu(z+1)\,\nabla.$

\item[II)] $\rho(L_{-1})=\frac1{\gamma}\cdot\beta_\mu(z)\,\nabla^{-1},\quad \rho(L_0)=z,\quad  \rho(L_{1})={\gamma}\cdot\alpha_\mu(z+1)\,\nabla.$

\item[III)] $\rho(L_{-1})=\frac1{\gamma}\cdot\alpha_\mu(z)\,\nabla^{-1},\quad \rho(L_0)=z,\quad  \rho(L_{1})={\gamma}\cdot\beta_\mu(z+1)\,\nabla.$
\item[IV)] $\rho(L_{-1})=\frac1{\gamma}\cdot\pi_\mu(z)\,\nabla^{-1},\quad \rho(L_0)=z,\quad  \rho(L_{1})={\gamma}\cdot\nabla.$
\end{enumerate} In all cases $\gamma$ is an arbitrary element of $\C^*$.
\end{thm}

\begin{proof} It is enough to take into account that a Casimir representation verifies $\rho(L_1)\circ \rho(L_{-1})=\pi_\mu(z+1)$. Letting $$\rho(L_{-1})=A_{-1}(z)\, \nabla^{-1},\quad \rho(L_{1})=A_{1}(z)\, \nabla,$$ this translates into $A_1(z)A_{-1}(z+1)=\pi_\mu(z+1)$ and the result follows.
\end{proof}

\begin{rem} According to Remark \ref{rem:number-smith-types}, the cardinality of the space $\mathcal S(1,\mu)$  of possible Smith types for Casimir representations of semi-level $\mu$ on a rank one free $\C[z]$-module is $4$. These correspond exactly to the four types of representations described in Theorem \ref{thm:casimir-reps-one-dim}. More precisely, representations of type $I)$, $II)$, $III)$, $IV)$ coorrespond to Smith types $S_0(0,1)$,  $S_+(0,1,0)$, $S_-(0,1,0)$, $S_0(1,0)$, respectively. Since we have made a choice of unique representatives for the equivalence classes of semi-levels that contains $\mu=0$, it follows that none of these representations are equivalent. \end{rem}

\begin{prop} The representations of Theorem \ref{thm:casimir-reps-one-dim}  of type $I)$ and $IV)$ are irreducible, whereas those of type $II)$ and $III)$ are reducible.
\end{prop}

\begin{proof} Let us suppose that $\rho$ is a representation of type $I)$ on $V\simeq \C[z]$. If $V'\subset V$ is an $\sl(2)$-submodule then it is invariant under $\rho(L_{-1})$ and therefore under the backward difference operator $\Delta_{-1}=\nabla^{-1}-\Id$.  Given $v\in V'$ we know that $\deg(\Delta_{-1}(v))=\deg(v)-1$. Hence, applying an appropriate power of $\Delta_{-1}$  to $v$ we conclude that $1\in V'$.  Now acting with $\rho(L_0)=z$ on $1\in V'$ we see that $\C[z]\simeq V\subset V'$, that is $V'=V$. A similar argument involving the forward difference operator applies to representations of type $IV)$.

Any representation of type $II)$ or $III)$ leaves invariant the subspace $V_{\geq k}\subset V\simeq \C[z]$ formed by the polynomials whose degree is greater or equal than a non negative integer $k$. Since $V_{\geq k}$ is a proper subspace of $V$ for $k\geq 1$, it follows that these representations are not irreducible.
\end{proof}

\begin{rem} The irreducible representations of type $I)$ and $IV)$ were first discovered by Arnal and Pinczon \cite{Arnal} and later put by Kostant in the broader context of Whittaker modules introduced by him in \cite{Kostant}. A recent study of $\sl(2)$-module structures on a rank one free $\C[z]$-module is in \cite{Nilsson} to be found.
\end{rem}


\section{Simple torsion free $\sl(2)$-modules of arbitrary rank}\label{sec:family}

The goal of this section is to introduce a whole family of simple torsion free $\sl(2)$-modules. More precisely, we will prove the following result.

\begin{thm}\label{thm:newsimpletorsionfree}
Let $\alpha= X^n -   p(z)X^{n-1} - a_0 \in {\mathbb A}^+$ with 
$\deg(p(z))\geq 1$ and $a_0\in\C\setminus\{0\}$. 

Then ${\mathbb A}/({\mathbb A}\alpha)$ is a simple torsion free $\sl(2)$-module of rank $n$. 
\end{thm}

Some previous results are required for the proof. 

%
%

\begin{prop}\label{prop:Smithtype1-submodule}
Let $V$ be a rank $n$ polynomial Casimir $\sl(2)$-representation of semi-level $\mu$, $V'$ be a rank $k$ polynomial $\sl(2)$-representation and $\psi:V'\hookrightarrow V$ an injective morphism of $\sl(2)$-modules. 

If $V'$ has Smith type $S_0(k,0)=(1, \overset{k)}\ldots,1)$, then $V/V'$ is torsion free.

If $V$ has Smith type $S_0(n,0)=(1, \overset{n)}\ldots,1)$, then $V/V'$ is a rank $n-k$ polynomial Casimir $\sl(2)$-representation of semi-level $\mu$ and $V'$ and $V/V'$ have  types $S_0(k,0)$ and $S_0(n-k,0)$ respectively. 
\end{prop}

\begin{proof}

Let $\{v'_1,\ldots, v'_k\}$ (resp. $\{v_1,\ldots, v_n\}$) be a basis of $V'$ (resp. $V$). Let $B(z)$ be the matrix associated to $\psi$ w.r.t. these basis. Following the notations and results of \S\ref{sec:polynomial}, let us identify $\rho_{V'}(L_1)$  with $A'_1(z)\circ \nabla$ acting on $\C[z]^k\simeq V'$ and $\rho_{V}(L_1)$ with $A_1(z)\circ \nabla$ acting on $\C[z]^n\simeq V$.
The fact that $\psi$ is a map of $\sl(2)$-modules yields
	\begin{equation}\label{eq:V'issubmodule}
	A_1(z) \circ \nabla \circ B(z) \, =\, B(z) A'_1(z)\circ \nabla
	\end{equation}

Assuming that $A'_1(z)$ has  Smith type $S_0(k,0)=(1, \overset{k)}\ldots,1)$, it follows that $ A'_1(z)$ is invertible in $M_{n\times k}(\C[z])$ and, thus, \eqref{eq:V'issubmodule} may be rewritten as  $A_1(z) B(z+1) A'_1(z) ^ {-1}= B(z) $. Let $b(z)$ be the greatest common divisor the $k\times k$-minors of $B(z)$ and note that  $b(z)\neq 0$ since $\psi$ is injective. From the previous identity and thanks to the generalized Cauchy-Binet formula for the minors of a product of matrices, it follows that there exists a polynomial $p(z)$ such that	
	$$
	p(z) b(z+1) \det(A'_1(z))^{-1} \,=\,  b(z) \qquad \text{in } \C[z]
	$$
for some polynomial $p(z)$. Thus  $b(z)\in \C\setminus\{0\}$ and, therefore, the cokernel of $\psi$ is free (as a $\C[z]$-module) and the claim is proved. 

Regarding the second claim,  let us now assume that  $A_1(z)$ has  Smith type $S_0(n,0)=(1, \overset{n)}\ldots,1)$. It implies that $A_1(z)\in\GL(n,\C[z])$ and, therefore equation~\eqref{eq:V'issubmodule} yields
	$$
	 B(z+1) \, =\, A_1(z)^{-1} B(z) A'_1(z) \qquad \text{in }M_{n\times k}(\C[z])
	$$
Analogously as above,  there exists a polynomial $p(z)$ such that	
	$$
	b(z+1)\,=\, p(z) b(z) \det(A'_1(z)) \qquad \text{in } \C[z]
	$$
and, thus, it follows that $\det(A'_1(z))\in\C\setminus\{0\}$. Recalling the definition of the Smith form, this condition implies that the Smith type of $V'$ is $S_0(k,0)=(1, \overset{k)}\ldots,1)$.

Note that the first part implies that $V/V'$ is torsion free and, by Theorem~\ref{thm:simpletorsionfreefingen-finiterankfree}, it is free as a $\C[z]$-module and therefore it is a  rank $n$ polynomial Casimir $\sl(2)$-representation of semi-level $\mu$. 

It remains to show that, under these conditions,  $V/V'$ has Smith type $S_0(n-k,0)=(1, \overset{n-k)}\ldots,1)$. The action on $\rho$ induces an action $\rho''$ on $V/V'\simeq \C[z]^{n-k}$. We write $\rho''(L_1)$ as $A''_1(z)\circ \nabla$. Let $C(z)\in M_{(n-k)\times k}(\C[z])$ be the matrix associated to the surjection $V\to  V/V'$. Then
	$$
	A''_1(z) \circ \nabla \circ C(z) \, =\, C(z) A_1(z)\circ \nabla
	$$
Observe that $\det (A_1(z))\in\C\setminus\{0\}$ since the Smith type of $V$ is $S_0(n,0)=(1, \overset{n)}\ldots,1)$. Moreover,  $c(z)$, the g.c.d. of the maximal minors of $C(z)$, is invertible since $V\to  V/V'$ is surjective. Arguing as above, there exists a polynomial $q(z)$ such that
	$$
	\det(A''_1(z)) c(z+1) q(z) \, =\, c(z)  \qquad \text{in } \C[z]
	$$
Hence, we obtain that $\det(A''_1(z))\in\C\setminus\{0\}$ and the claim is proved. 
\end{proof}

\begin{cor}
The full subcategory of $\sl(2)$-modules consisting of polynomial Casimir modules of  Smith type $S_0(n,0)$ for $n>0$ is an abelian category. Further, it is a Krull-Schmidt category.  
\end{cor}

\begin{proof}
It follows easily from the previous Proposition and from the fact that the objects have finite length.
\end{proof}

\begin{prop}\label{prop:irred-rest} 
Let $\alpha=a_0(z)+a_1(z) X+\cdots+a_n(z) X^n\in\mathbb A^+$, where $n\geq 0$. 
If $a_0(z), a_n(z)\in\C\setminus\{0\}$, then $V_\alpha=\mathbb A/(\mathbb A\cap\mathbb B\alpha)$  is a free $\C[z]$-module of rank $n$ isomorphic to 
$\mathbb A/(\mathbb A\alpha)$ and it is a Casimir polynomial representation of semi-level $\mu$ and Smith type $S_0(n,0)=(1, \overset{n)}\ldots,1)$.
\end{prop}

\begin{proof} If $a_0(z), a_n(z)\in\C$, then by  Proposition~\ref{prop:submodule}  one has that $V_{\alpha}\simeq  \mathbb A/(\mathbb A\alpha) $ is a rank $n$ polynomial representation. 

Now, using the notations of the proof of Proposition~\ref{prop:Smithtype1-submodule} with respect to the basis $1,X,\ldots, X^{n-1}$,  it follows that $A_1(z) $ is a companion matrix; namely
	{\small $$
	A_1(z)\,=\, { \footnotesize\begin{pmatrix} 
		0 & & \dots &  -\frac{a_0(z)}{a_n(z)}  
		\\ 
		1 & 0 &   &    0 
		 \\ 
		 \vdots  & & \ddots  & \vdots \\ 
		 0 & \dots &  1 & -\frac{a_{n-1}(z)}{a_n(z)} 
		 \end{pmatrix}}\,\in\, M_{n\times n}(\C[z])
	$$}and, thus, $\det(A_1(z))\in \C\setminus\{0\}$. Having in mind that the invariant factors of $A_1(z)$ divide $\det(A_1(z))$ and applying Theorem~\ref{thm:polynomialmuCasimir}, it follows that $V$ is a Casimir polynomial representation of semi-level $\mu$ and Smith type $S_0(n,0)=(1, \overset{n)}\ldots,1)$.
\end{proof}

\begin{proof}[Proof of Theorem~\ref{thm:newsimpletorsionfree}]
By Proposition~\ref{prop:submodule}  we know that  $V:= {\mathbb A}/({\mathbb A}\alpha)$ is a  free $\C[z]$-module of rank $n$. 
Since $\C[z]$ is a p.i.d. and $V$ is free of finite rank, it follows that any $\C[z]$-submodule of $V$ is free of finite rank. Hence, let us assume that there exists a non-trivial $\sl(2)$-submodule $V'\simeq \C[z]^k \hookrightarrow V:= {\mathbb A}/{\mathbb A}\alpha$ with $k\leq n$. Now, Proposition~\ref{prop:Smithtype1-submodule} shows that $V'$ has Smith type $(1,\overset{k)}\ldots, 1)$ and, in particular, $\det(A'_1(z)) \in\C\setminus\{0\}$ where $\rho(L_1)\vert_{V'}$ is expressed as $A'_1(z)\circ\nabla$ w.r.t. a basis. 

We also obtain an analogous to the equation~\eqref{eq:V'issubmodule}; namely, the matrix identity $A_1(z) B(z+1) = B(z)A'_1(z)$ where $A_1(z)$ is the companion matrix of $\alpha$ w.r.t. the basis $v_1:=1,v_2:=X,\ldots, v_{n}:=X^{n-1}$
		{\small $$
	A_1(z)\,=\, { \footnotesize\begin{pmatrix} 
		0 & & \dots &  a_0 
		\\ 
		1 & \ddots &   &    0 
		 \\ 
		 \vdots  & & 0  & \vdots \\ 
		 0 & \dots &  1 & p(z)
		 \end{pmatrix}}\,\in\, M_{n\times n}(\C[z])
	$$}(since $\rho(L_1)$ is the left multiplication by $X$) and $B(z)$ is the matrix associated to $V'\hookrightarrow V $ 
	{\small $$
	B(z)\,=\, { \footnotesize\begin{pmatrix} 
	b_{11}(z) &  \ldots & b_{1k}(z) \\  
	\vdots & & \vdots \\  
	b_{n1}(z) &  \ldots & b_{nk}(z)  \end{pmatrix}}
	$$}

The $k$-th exterior product of the matrix identity $A_1(z) B(z+1) = B(z)A'_1(z)$, gives	
	\begin{equation}\label{eq:wedgekA1}
	\wedge^k A_1(z) \Big(\sum_{I\in {\mathcal I}} b_I(z+1) v_I\Big) \,=\, \lambda \Big(\sum_{I\in {\mathcal I}} b_I(z) v_I\Big) 
	\end{equation}
where ${\mathcal I}$ is the set of multi-indexes $ i_1<\ldots < i_k$ with $1\leq i_1, i_k\leq n$, $ b_I(z)$ denotes the minor of $B(z)$ corresponding to the rows $i_1,\ldots, i_k$,  $v_I := v_{i_1}\wedge\ldots\wedge v_{i_k}$ for $I\in{\mathcal I}$ and $v_I := 0$ for any other multi-index not lying in ${\mathcal I}$, and $\lambda=\det(A'_1(z))$. 

Note that $\wedge^kA_1(z)$ acts as follows
	{\small $$
	\big(\wedge^kA_1(z)\big) v_I\,=\,
	\begin{cases}
	v_{\sigma(I)}  & \text{ for } i_k< n
	\\
	(-1)^{k-1} a_0 v_{\sigma(I)}+p(z) v_{\tau(I)} & \text{ for } i_k= n
	\end{cases}
	$$}
where $I\equiv i_1<\ldots <i_k$, 
	$$
	\begin{aligned}
	\sigma:\, &{\mathcal I}\overset{\sim}\to{\mathcal I} \\
	& I \mapsto \, \sigma(I):= 
		\begin{cases} i_1+1<\ldots <i_k+1 & \text{ for } i_k< n \\ 
		1<i_1+1<\ldots< i_{k-1}+1 & \text{ for } i_k=n  \end{cases}
	\end{aligned}
	$$
and 
	$$
	\begin{aligned}
	\tau:\, &{\mathcal I}\to{\mathcal I}\cup \{0\} \\
	& I \mapsto \, \tau (I):= 
		\begin{cases} i_1+1<\ldots <i_{k-1}+1<n & \text{ for } i_{k-1}<n-1, i_k=n,  \\ 
		0 & \text{ otherwise}  \end{cases}
	\end{aligned}
	$$

Looking at the coefficients of $v_I$ in the identity~\eqref{eq:wedgekA1}, we obtain 	
	$$
	\pm a_0\,b_{\sigma^{-1}(I)}(z+1) \, +\, p(z) b_{\tau^{-1}(I)}(z+1)  \,=\, \lambda b_I(z) 
	$$
where,  by convention we set $b_{\emptyset}=0$. Denote $d_I:=\operatorname{deg}(b_I(z))$ for $I\in {\mathcal I}$ and $d_{\emptyset} := -\infty$. Then, the previous equation yields  
	$$
	\begin{aligned}
	d_{\sigma^{-1}(I)} \,=\, d_I \qquad & \text{ if } \tau^{-1}(I)=\emptyset
	\\
	 d_{J} \,<\, \operatorname{deg}(p(z)) + d_{J} \,\leq\, \operatorname{max}\{ d_I, d_{\sigma^{-1}(I)}\}  
	\qquad& \text{ if }  \tau(J)=I
	\end{aligned}
	$$
For $I\equiv i_1<\ldots < i_k \in {\mathcal I}$, we set $I(j):=i_j$. The relations above imply the following properties
	\begin{enumerate}
		\item if $d_I$ is maximal in the set $\{d_I\vert I\in{\mathcal I}\}$, then $\tau(I)=0$ and, therefore, either $I(n)<n$ or $I(n-1)=n-1$;
		\item if $I(1)\geq 2$ and $I(n)<n$, then $\tau^{-1}(I)=\emptyset$ and, thus, $d_I=d_{\sigma^{-1}(I)}$;
		\item if $I(1)=1$, then $\tau^{-1}(I)=\emptyset$ and, thus, $d_I=d_{\sigma^{-1}(I)}$.
	\end{enumerate}

Let $I$ be a multiindex $I\equiv i_1<\ldots <i_k$ such that $d_I$ is maximal among all $d$'s. By (1) it holds that either $I(n)<n$ or $I(n-1)=n-1$. 

Let us deal with the first case; that is $I(n)<n$. Applying (2) $i_1-1$ times, one has that $d_{\sigma^{-(i_1+1)}(I)}$ is maximal, or, what is tantamount, we may assume that the maximum is attained at some $I$ with $I(1)=1$. Item (3) shows that   $d_I=d_{\sigma^{-1}(I)}$. Noting that $\sigma^{-1}(I)(n)=n$ and having in mind (1), it follows that $\sigma^{-1}(I)(n-1)=n-1$. On the other hand, $\sigma^{-1}(I)(n-1)=I(n)-1$ and, hence, $I(n)=n$. That is,  the maximum is also attained at some $I$ with $I(n)=n$. 

Hence, let us now deal with the case $I(n)=n$; and, by (1), $I(n-1)=n-1$.  Observe that $\sigma(I)(1)=1$, $\sigma(I)(n)=n$ and   $\tau^{-1}(\sigma(I))=\emptyset$. 
Thus,   $d_I=d_{\sigma(I)}$ is maximal and, by (1),  $\sigma(I)(n-1)=n-1$. On the other hand $\sigma(I)(n-1)=I(n-2)+1$; i.e. $I(n-2)=n-2$. Applying this recursively, one obtains that $\sigma^{k-1}(I)=(1,\ldots,k-1,n)$ and $d_{\sigma^{k-1}(I)}$ is maximal. By (1) we get $k=n$ and thus $I(j)=j$ for all $1\leq j\leq n$. Therefore  $V'\simeq V$ since by Proposition \ref{prop:Smithtype1-submodule} it follows that $V/V'$ is a torsion free $\C[z]$-module of rank $0$. 
\end{proof}




\begin{thebibliography}{99}

\bibitem{Arnal} Arnal, D., Pinczon G.; On algebraically irreducible representations of the Lie algebra $\sl(2)$, J. Math. Phys. 15  (1974) no. 3,  pp. 350--359.

\bibitem{Bavula1}  Bavula, V. V. Generalized Weyl algebras and their representations. Algebra i Analiz 4 (1992), no. 1, pp. 75--97.

\bibitem{Bavula2}  Bavula, V. V. Classification of simple $\sl(2)$-modules and the finite-dimensionality of the module of extensions of simple $\sl(2)$-modules. Ukrain. Mat. Zh. 42 (1990), no. 9, pp. 1174--1180.
 
 \bibitem{Block} Block, R. E. The irreducible representations of the Lie algebra $\sl(2)$ and of the Weyl algebra. Adv. in Math. 39 (1981), no. 1, pp. 69--110.

\bibitem{Hazewinkel} Hazewinkel, M., Gubareni, N., Kirichenko, V. V.; Algebras, rings and modules. Vol. 1. Mathematics and its Applications, 575. Kluwer Academic Publishers, Dordrecht, 2004.

\bibitem{Kostant} Kostant, B. On Whittaker Vectors and Representation Theory. Invent. Math. 48 (1978), no. 2, pp.101-184.



\bibitem{Mazor}  Mazorchuk, V.,  Lectures on $\sl_2(\C)$-modules. Imperial College Press, London, 2010.  

\bibitem{Newman} Newman, M. Integral matrices. Pure and Applied Mathematics, Vol. 45. Academic Press, New York-London, 1972.

\bibitem{Nilsson} Nilsson, J. . Simple $\sl_{n+1}$-module structures on $U(h)$, Journal of Algebra, Vol. 424 (2015), pp. 294--329 


\bibitem{Plaza-Tejero} Plaza Martin, F. J., Tejero Prieto, C.  Extending Representations of $\sl(2)$ to Witt and Virasoro algebras, arXiv:1411.5598. 

\bibitem{Puninskii} Puninski{\u\i} , G. E. Left almost split morphisms and generalized Weyl algebras, Math. Notes {66} (1999), no. {5-6}, pp. {608--612 (2000)}.

\end{thebibliography}
\end{document}